\documentclass[12pt,twoside]{amsart}

\usepackage{amssymb}
\usepackage{graphicx}
\usepackage[pdfauthor={David Cook II, Uwe Nagel},
            pdftitle={Some algebras with the weak Lefschetz property},
            pdfsubject={Commutative algebra and enumerative combinatorics},
            pdfkeywords={Monomial ideals, weak Lefschetz property, lozenge tilings},
            pdfproducer={LaTeX with hyperref},
            pdfcreator={latex->dvips->ps2pdf},
            pdfpagemode=UseOutlines,
            bookmarksopen=false,
            letterpaper,
            bookmarksnumbered=true,
            pdfborder={0 0 0},
            colorlinks=false]{hyperref}
\usepackage[margin=1in]{geometry}

\newcommand{\dntri}{\bigtriangledown}
\newcommand{\uptri}{\triangle}

\newcommand{\mo}{\mathfrak{o}}

\newcommand{\PP}{\mathbb{P}}
\newcommand{\ZZ}{\mathbb{Z}}

\DeclareMathOperator{\per}{perm}

\DeclareMathOperator{\syz}{syz}

\DeclareMathOperator{\soc}{soc}

\newcommand{\st}{\; \mid \;}
\def\urltilda{\kern -.15em\lower .7ex\hbox{\~{}}\kern .04em}

\numberwithin{figure}{section}
\numberwithin{equation}{section}

\newtheorem{theorem}{Theorem}[section]
\newtheorem{lemma}[theorem]{Lemma}
\newtheorem{proposition}[theorem]{Proposition}
\newtheorem{corollary}[theorem]{Corollary}

\theoremstyle{definition}

\newtheorem{remark}[theorem]{Remark}
\newtheorem{example}[theorem]{Example}

\begin{document}

% -- Title Block and Abstract
\title{Some algebras with the weak Lefschetz property}

\author[D.\ Cook II]{David Cook II}
\address{Department of Mathematics \& Computer Science, Eastern Illinois University, Charleston, IL 46616}
\email{\href{mailto:dwcook@eiu.edu}{dwcook@eiu.edu}}

\author[U.\ Nagel]{Uwe Nagel}
\address{Department of Mathematics, University of Kentucky, 715 Patterson Office Tower, Lexington, KY 40506-0027}
\email{\href{mailto:uwe.nagel@uky.edu}{uwe.nagel@uky.edu}}

\thanks{    
The second author was  partially supported by the National Security Agency under Grant Number H98230-12-1-0247
    and by the Simons Foundation under grants \#208869 and \#317096.}


\maketitle

\begin{center}
{\em Dedicated to Winfried Bruns at the occasion of his $70^{th}$ birthday.}
\end{center}

\begin{abstract}
    Using a connection to lozenge tilings of triangular regions, we establish an easily checkable criterion that guarantees  the weak Lefschetz property of a quotient by a  monomial ideal.  It is also shown that each such ideal also has a  semistable syzygy bundle.
\end{abstract}

%----------------
% -- Introduction
\section{Introduction} \label{sec:intro}

Recently, there have been many investigations of the presence of the weak Lefschetz property (see, e.g., \cite{BMMNZ, BMMNZ2, BK,  HSS,  KRV,  LZ, MMN, MMN-2012,MN-survey, Stanley-1980}).
A standard graded Artinian algebra $A$ over a field $K$ is said to have the \emph{weak Lefschetz property}
if there is a linear form $\ell \in A$ such that the multiplication map $\times \ell : [A]_i \rightarrow [A]_{i+1}$ has maximal rank
for all $i$ (i.e., it is injective or surjective).   The name is a reminder that
the Hard Lefschetz Theorem guarantees this property for the cohomology ring of a compact K\"ahler manifold. It is a desirable property, e.g., 
    its presence constrains the Hilbert function (see~\cite{HMNW}).
Many algebras are expected to have the weak Lefschetz property. However, establishing this fact is often very challenging.  
The recent lecture notes of Harima, Maeno, Morita, Numata, Wachi, and Watanabe~\cite{HMMNWW} and the survey \cite{MN-survey}  provide an excellent review of the 
state of knowledge on the Lefschetz properties. 

The authors have developed a combinatorial approach towards deciding the presence of the weak Lefschetz property for monomial algebras in three variables in \cite{CN-resolutions, CN-small-type}. It relies on a study of lozenge tilings of so-called triangular regions. The needed results are recalled in  Subsection~\ref{sub:tri}. This approach has been used, for example,   to investigate the  weak Lefschetz property  of monomial algebras of type two  \cite{CN-small-type} and quotients by  ideals with fours generators \cite{CN-amaci}. Furthermore, we showed in \cite{CN-amaci} that there is a connection between lozenge tilings and semistability of  syzygy bundles  (see Theorem~\ref{thm:wlp-biadj}). Here we use these methods to establish sufficient conditions for the presence of the weak Lefschetz property as well as the semistability of the syzygy bundle (see Theorem \ref{thm:wlp-iff-semistab}). These conditions are easily expressed and checked using triangular regions (see Remark~\ref{rem:pictural description}). As a consequence, we describe infinite families of ideals which satisfy these conditions. The number of generators of these ideals can be arbitrarily large. As a simple example, consider the following ideal 
\[
I = (x^{12}, \ x^6 y^2 z^3, \ x^3 y^2 z^7, \ x y^7 z^3, \ x y^5 z^5, \ x y z^9, \ y^{12}, \ z^{12}). 
\]
Our criterion immediately gives that its quotient has the weak Lefschetz property.

%----------------------
% -- Algebraic triangular regions
\section{Algebraic triangular regions}
     \label{sec:trireg} 

We recall facts needed to establish our main results in the following section. 

Let $R = K[x,y,z]$ be the standard graded polynomial ring over the field $K$. Unless specified otherwise, $K$ is always a field of characteristic zero. All $R$-modules are assumed to be finitely generated and graded. 

Let $I$ be a monomial ideal of $R$. Then  $A = R/I = \oplus_{j \ge 0} [A]_j$ is the sum of  finite vector spaces, called the \emph{homogeneous components (of $A$) of degree $j$}.
The \emph{Hilbert function} of $A$ is the function $h_A: \ZZ \to \ZZ$ given by $h_A(j) = \dim_K [A]_j$. The
\emph{socle} of $A$, denoted $\soc{A}$, is the annihilator of $\mathfrak{m} = (x, y, z)$, the homogeneous
maximal ideal of $R$, that is, $\soc{A} = \{a \in A \st a \cdot \mathfrak{m} = 0\}$. 

%-- Triangular regions
\subsection{Triangular regions labeled by monomials}\label{sub:tri}~

Now we briefly review a connection between monomial ideals and triangular regions; for a more thorough discussion see~\cite{CN-resolutions}.

For an integer $d \geq 1$, the triangular region is the \emph{triangular region (of $R$) in degree $d$}, denoted $\mathcal{T}_d$,
is an equilateral triangle of side length $d$ composed of $\binom{d}{2}$ downward-pointing ($\dntri$) and $\binom{d+1}{2}$ 
upward-pointing ($\uptri$) equilateral unit triangles.  These triangles are labeled by the monomials in $[R]_{d-2}$ and $[R]_{d-1}$, respectively,
as follows:  place $x^{d-1}$ at the top $y^{d-1}$ at the bottom-left, and $z^{d-1}$ at the bottom-right; the remaining
labels are found via interpolation. See Figure~\ref{fig:triregion-R}(i) for an illustration.

\begin{figure}[!ht]
    \begin{minipage}[b]{0.48\linewidth}
        \centering
        \includegraphics[scale=1]{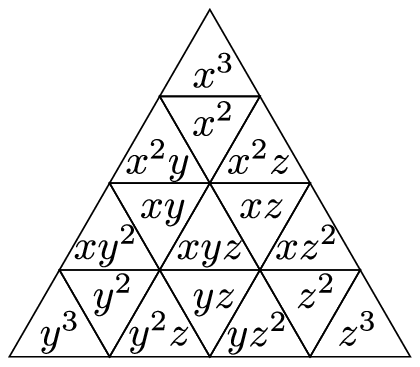}\\
        \emph{(i) $\mathcal{T}_4$}
    \end{minipage}
    \begin{minipage}[b]{0.48\linewidth}
        \centering
        \includegraphics[scale=1]{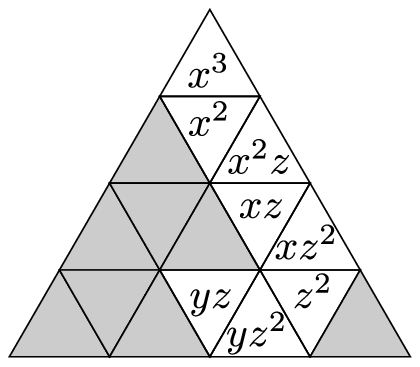}\\
        \emph{(ii) $T_4(xy, y^2, z^3)$}
    \end{minipage}
    \caption{A triangular region with respect to $R$ and with respect to $R/I$.}
    \label{fig:triregion-R}
\end{figure}

The \emph{triangular region (of $R/I$) in degree $d$}, denoted by $T_d(I)$, is the part of $\mathcal{T}_d$ that is obtained
after removing the triangles labeled by monomials in $I$.  Note that the labels of the downward- and
upward-pointing triangles in $T_d(I)$ form $K$-bases of $[R/I]_{d-2}$ and $[R/I]_{d-1}$, respectively.
See Figure~\ref{fig:triregion-R}(ii) for an example. 

Notice that the regions missing from $\mathcal{T}_d$ in $T_d(I)$ can be viewed as a union of (possibly overlapping)
upward-pointing triangles of various side lengths that include the upward- and downward-pointing triangles inside them.
Each of these upward-pointing triangles corresponds to a minimal generator $x^a y^b z^c$ of $I$ that has, necessarily, degree at most
$d-1$.  The value $d-(a+b+c)$ is the \emph{side length of the puncture associated to $x^a y^b z^c$}, regardless of possible
overlaps with other punctures.  See Figure~\ref{fig:triregion-punctures} for an example.

\begin{figure}[!ht]
    \begin{minipage}[b]{0.48\linewidth}
        \centering
        \includegraphics[scale=1]{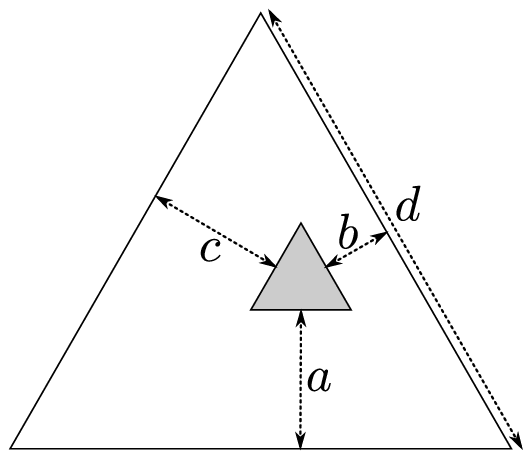}\\
        \emph{(i) $T_{d}(x^a y^b z^c)$}
    \end{minipage}
    \begin{minipage}[b]{0.48\linewidth}
        \centering
        \includegraphics[scale=1]{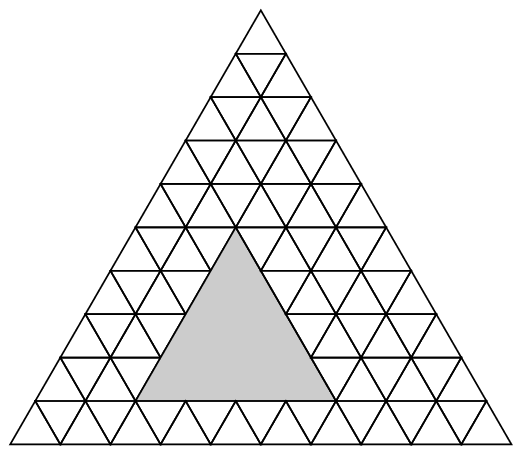}\\
        \emph{(ii) $T_{10}(xy^3z^2)$}
    \end{minipage}
    \caption{$T_d(I)$ as constructed by removing punctures.}
    \label{fig:triregion-punctures}
\end{figure}

We say that two punctures \emph{overlap} if they share at least an edge. Two punctures are said to be \emph{touching}
if they share precisely a vertex.

\newpage

%--
\subsection{Tilings \& the weak Lefschetz property}\label{sub:tiling}~

A \emph{lozenge} is a union of two unit equilateral triangles glued together along a shared edge, i.e., a rhombus with
unit side lengths and angles of $60^{\circ}$ and $120^{\circ}$. Lozenges are also called calissons and diamonds in the
literature.  See Figure~\ref{fig:triregion-intro}.

\begin{figure}[!ht]
    \begin{minipage}[b]{0.48\linewidth}
        \centering
        \includegraphics[scale=1]{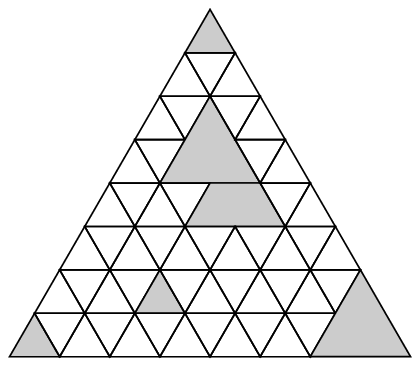}
    \end{minipage}
    \begin{minipage}[b]{0.48\linewidth}
        \centering
        \includegraphics[scale=1]{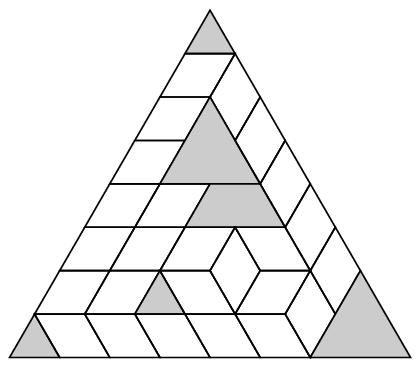}
    \end{minipage}
    \caption{A triangular region $T \subset \mathcal{T}_8$ together with on of its $13$ tilings.}
    \label{fig:triregion-intro}
\end{figure}

Fix a positive integer $d$ and consider the triangular region $\mathcal{T}_d$ as a union of unit triangles. Thus a \emph{subregion}
$T \subset \mathcal{T}_d$ is a subset of such triangles. We retain their labels. As above, we say that
a subregion $T$ is \emph{$\dntri$-heavy}, \emph{$\uptri$-heavy}, or \emph{balanced} if there are more downward pointing
than upward pointing triangles or less, or if their numbers are the same, respectively. A subregion is \emph{tileable}
if either it is empty or there exists a tiling of the region by lozenges such that every triangle is part of exactly one
lozenge.  Since a lozenge in $\mathcal{T}_d$ is the union of a downward-pointing and an upward-pointing triangle, and every
triangle is part of exactly one lozenge, a tileable subregion is necessarily balanced.

Let $T \subset \mathcal{T}_d$ be any subregion. Given a monomial $x^a y^b z^c$ with degree less than $d$, the
\emph{monomial subregion} of $T$ associated to $x^a y^b z^c$ is the part of $T$ contained in the triangle $a$ units from
the bottom edge, $b$ units from the upper-right edge, and $c$ units from the upper-left edge. In other words, this
monomial subregion consists of the triangles that are in $T$ and the puncture associated to the monomial $x^a y^b z^c$.

We previously established a characterization of tileable triangular regions associated to monomial ideals.

\begin{theorem}{\cite[Theorem 2.2]{CN-resolutions}}\label{thm:tileable}
    Let $T = T_d(I)$ be a balanced triangular region, where $I \subset R$ is any monomial ideal.  Then $T$ is tileable if and only if
    $T$ has no $\dntri$-heavy monomial subregions.
\end{theorem}

A subregion $T \subset \mathcal{T}_d$ can be associated to a bipartite planar graph $G$ that is an induced subgraph
of the honeycomb graph. Lozenge tilings of $T$ can be then related to perfect matchings on $G$. The connection was
used by Kuperberg in~\cite{Kup}, the earliest citation known to the authors, to study symmetries on plane partitions.

Using this connection, the \emph{bi-adjacency matrix of $T$} is the bi-adjacency matrix $Z(T) := Z(G)$ of the graph $G$ associated to $T$.
If $T = T_d(I)$ for some Artinian ideal $I$, then $Z(T)$ is the transpose of the matrix defined by $\times(x+y+z): [R/I]_{d-2} \to [R/I]_{d-1}$ using monomial bases in the reverse-lexicographic order (see  \cite[Proposition 4.5]{CN-small-type}).
Based on results by Migliore, Mir\'o-Roig, and the second author~\cite{MMN}, we established the following criterion for
the presence of the weak Lefschetz property. 

\begin{theorem}{\cite[Corollary 4.7]{CN-small-type}}\label{thm:wlp-biadj}
    Let $I$ be an Artinian monomial ideal in $R = K[x,y,z]$. Then $R/I$ has the weak Lefschetz property if and only if,
    for each positive integer $d$, the matrix $Z(T_d(I))$ has maximal rank.
\end{theorem}

It is well-known that the permanent of $Z(G)$ enumerates the perfect matchings of $G$.  
A perfect matching of $G$ can be signed via the permutation it generates; thus a lozenge tiling $\tau$ of $T$ can be similarly signed (see \cite{CN-resolutions}).
Hence  the \emph{signed} tilings of the region are related to the determinant of the bi-adjacency matrix.

\begin{theorem}{\cite[Theorem 3.5]{CN-resolutions}}\label{thm:pm-det}
     If $T \subset \mathcal{T}_d$ is a non-empty balanced subregion, then the signed lozenge tilings of
     $T$ are enumerated by $\det{Z(T)}$.
\end{theorem}

We recursively define a puncture of $T \subset \mathcal{T}_d$ to be a \emph{non-floating} puncture if it touches the boundary of 
$\mathcal{T}_d$ or if it overlaps or touches a non-floating puncture of $T$. Otherwise we call a puncture a \emph{floating} puncture.
We further have the regions with only even floating punctures have an easier enumeration.

\begin{corollary}{\cite[Corollary 4.7]{CN-resolutions}}\label{cor:same-sign}
    Let $T$ be a tileable triangular region, and suppose all floating punctures of $T$ have an even side length. Then
    every lozenge tiling of $T$ has the same perfect matching sign as well as the same lattice path sign, and so
    $\per{Z(T)} = |\det{Z(T)}| > 0$.
\end{corollary}

%--
\subsection{Tileability \& semistability of the syzygy bundle}\label{sub:syz}~

Let $I$ be a monomial ideal of $R$ whose punctures in $\mathcal{T}_d$ (corresponding to the minimal generators of $I$ having degree less than $d$)
have side lengths that sum to $m$. Then we define the \emph{over-puncturing coefficient} of $I$ in degree $d$ to be $\mo_d (I) = m - d$.
If $\mo_d (I) < 0$, $\mo_d (I) = 0$, or $\mo_d (I) > 0$, then we call $I$ \emph{under-punctured}, \emph{perfectly-punctured}, or
\emph{over-punctured} in degree $d$, respectively.

Let now $T = T_d(I)$ be a triangular region with punctures whose side lengths sum to $m$.  Then we define similarly the \emph{over-puncturing coefficient}
of $T$ to be $\mo_T = m - d$.  If $\mo_T < 0$, $\mo_T = 0$, or $\mo_T > 0$, then we call $T$ \emph{under-punctured}, \emph{perfectly-punctured},
or \emph{over-punctured}, respectively. Note that $\mo_T = \mo_d (J(T)) \leq \mo_d (I)$, and equality is true if and only if the ideals $I$ and
$J(T)$ are the same in all degrees less than $d$.

Perfectly-punctured regions admit a numerical tileability criterion.

\begin{proposition}{\cite[Corollary 2.4]{CN-amaci}} \label{pro:pp-tileable}
    Let $T = T_d(I)$ be a triangular region.  Then any two of the following conditions imply the third:
    \begin{enumerate}
        \item $T$ is perfectly-punctured;
        \item $T$ has no over-punctured monomial subregions; and
        \item $T$ is tileable.
    \end{enumerate}
\end{proposition}

Let $I$ be an Artinian ideal of $S = K[x_1, \ldots, x_n]$ that is minimally generated by forms $f_1, \ldots, f_m$.
The \emph{syzygy module} of $I$ is the graded module $\syz{I}$ that fits into the exact sequence
\[
    0 \rightarrow \syz{I} \rightarrow \bigoplus_{i=1}^{m}S(-\deg f_i) \rightarrow I \rightarrow 0.
\]
Its sheafification $\widetilde\syz{I}$ is a vector bundle on $\PP^{n-1}$, called the \emph{syzygy bundle} of $I$. It has rank $m-1$.

Let $E$ be a vector bundle on projective space.  The \emph{slope} of $E$ is defined as $\mu(E) := \frac{c_1(E)}{rk(E)}$.  Furthermore, 
$E$ is said to be \emph{semistable} if the inequality $\mu(F) \leq \mu(E)$ holds for every coherent subsheaf $F \subset E$.  

Using the characterization of semistability for monomial ideals given by Brenner~\cite{Br}, we previously established a  connection to tileability.

\begin{theorem}{\cite[Theorem 3.3]{CN-amaci}} \label{thm:tileable-semistable}
    Let $I$ be an Artinian ideal in $R = K[x,y,z]$ generated by monomials whose degrees are bounded above by $d$,
    and let $T = T_d(I)$.  If $T$ is non-empty, then any two of the following conditions imply the third:
    \begin{enumerate}
        \item $I$ is perfectly-punctured;
        \item $T$ is tileable; and
        \item $\widetilde\syz{I}$ is semistable.
    \end{enumerate}
\end{theorem}

We note that Brenner and Kaid showed in \cite[Corollary~3.3]{BK} that, for almost complete intersections, nonsemistability implies the weak
Lefschetz property in characteristic zero.

% -----------------------------------------------
\section{The criterion}

In this section we will establish sufficient conditions that guarantee the presence of the weak Lefschetz property. We use this result to exhibit explicit infinite families of ideals such that their quotients have the weak Lefschetz property.

We begin with a  necessary condition for the presence of the weak Lefschetz property.

\begin{proposition} \label{pro:non-tileable-non-wlp}
    Let $I$ be a monomial ideal such that $T_d(I)$ is a balanced region that is not tileable. Put $J = I + (x^d, y^d, z^d)$.
    Then $R/J$ never has the weak Lefschetz property, regardless of the characteristic of $K$.
\end{proposition}
\begin{proof}
    Since $T_d(I) = T_d(J)$ is not tileable, Theorem~\ref{thm:pm-det} gives $\det Z(T_d(J)) = 0$. Thus, $Z(T_d(J))$
    does not have maximal rank. Now we conclude by Theorem~\ref{thm:wlp-biadj}.
\end{proof}

We illustrate the preceding proposition with an example.

\begin{example}
    Consider the regions depicted in Figure~\ref{fig:nontileable}.
    \begin{figure}[!ht]
        \begin{minipage}[b]{0.48\linewidth}
            \centering
            \includegraphics[scale=1]{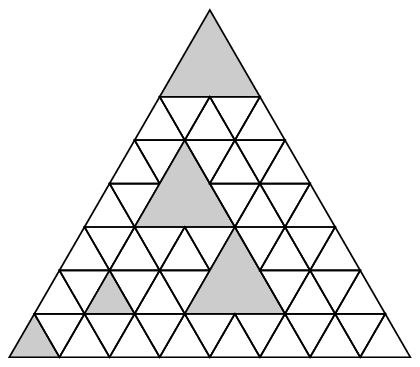}\\
            \emph{(i) $T = T_8(x^6, y^7, z^8, xy^5z, xy^2z^3, x^3y^2z)$}
        \end{minipage}
        \begin{minipage}[b]{0.48\linewidth}
            \centering
            \includegraphics[scale=1]{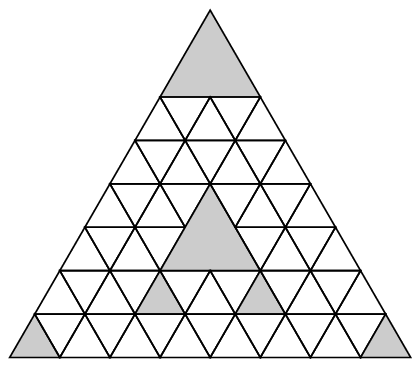}\\
            \emph{(ii) $T' = T_8(x^6, y^7, z^7, xy^4z^2, xy^2z^4, x^2y^2z^2)$}
        \end{minipage}
        \caption{Two balanced non-tileable triangular regions.}
        \label{fig:nontileable}
    \end{figure}
    These regions are both balanced, but non-tileable as they contain $\dntri$-heavy monomial subregions (see
    Theorem~\ref{thm:tileable}). In particular, the monomial subregion associated to $xy^2z$ in $T$ and the monomial
    subregion associated to $xy^2z^2$ in $T'$ are both $\dntri$-heavy. Thus, $R/ (x^6, y^7, z^8, xy^5z, xy^2z^3, x^3y^2z)$
    and $R/(x^6, y^7, z^7, xy^4z^2, xy^2z^4, x^2y^2z^2)$ both fail to have the weak Lefschetz property,
    regardless of the characteristic of the base field.
\end{example}

Now we use Proposition~\ref{pro:non-tileable-non-wlp} in order to relate the weak Lefschetz property and semistability of
syzygy bundles (see Section~\ref{sub:syz}). In preparation, we record the following observation. Recall that the
monomial ideal of a triangular region $T \subset {\mathcal T}_d$ is the largest ideal $J$ whose minimal generators have
degrees less than $d$ such that $T = T_d (J)$.

\begin{lemma}\label{lem:lcm-overlap}
    Let $J \subset R$ be the monomial ideal of a triangular region $T \subset {\mathcal T}_d$. Then:
    \begin{enumerate}
        \item The region $T$ has no overlapping punctures if and only if each degree of a least common multiple of two
            distinct minimal generators of $J$ is at least $d$.
        \item The punctures of $T$ are not overlapping nor touching if and only if each degree of a least common
            multiple of two distinct minimal generators of $J$ is at least $d+1$.
    \end{enumerate}
\end{lemma}
\begin{proof}
    Let $m_1$ and $m_2$ be two distinct minimal generators of $J$. Then their associated punctures overlap if and only
    if there is a monomial of degree $d-1$ that is a multiple of $m_1$ and $m_2$. The existence of such a monomial means
    precisely that the degree of the least common multiple of $m_1$ and $m_2$ is at most $d-1$. Now claim (i) follows.

    Assertion (ii) is shown similarly by observing that the punctures to $m_1$ and $m_2$ touch if and only if there is a 
    monomial of degree $d$ that is a multiple of $m_1$ and $m_2$.
\end{proof}

The following consequence is useful later on.

\begin{corollary}\label{cor:socle-degree-bound}
    Assume $T \subset {\mathcal T}_d$ is a triangular region whose punctures are not overlapping nor touching, and let $J$
    be the monomial ideal of $T$. Then $R/J$ does not have non-zero socle elements of degree less than $d-1$.
\end{corollary}

\begin{proof}
    Since $J$ is a monomial ideal, every minimal first syzygy of $J$ corresponds to a relation $m_i n_i - m_j n_j = 0$ for
    suitable monomials $n_i$ and $n_j$, where $m_i$ and $m_j$ are distinct monomial minimal generators of $J$. Applying
    Lemma~\ref{lem:lcm-overlap} to the equality $m_i n_i = m_j n_j$, we conclude that the degree of each first syzygy of
    $J$ is at least $d+1$. It follows that the degree of every second syzygy of $J$ is at least $d+2$. Each minimal
    second syzygy of $J$ corresponds to a socle generator of $R/J$.
    Hence, the degrees of the socle generators of $R/J$ are at least $d-1$ .
\end{proof}

\begin{remark}
The converse of Corollary~\ref{cor:socle-degree-bound} is not true in general. For example, the socle generators of
$R/(x^6, y^7, z^8, xy^5z, xy^2z^3, x^3y^2z)$ have degrees greater than 7, but two punctures of
$T_8 (x^6, y^7, z^8, xy^5z, xy^2z^3, x^3y^2z)$ touch each other (see Figure~\ref{fig:nontileable}).
\end{remark}

Recall that perfectly-punctured regions were defined above Proposition~\ref{pro:pp-tileable}. This concept is used in
the proof of the following result.

\begin{theorem}\label{thm:wlp-to-semistab}
    Let $I \subset R$ be an Artinian ideal whose minimal monomial generators have degrees  $d_1,\ldots,d_t$. Set
    \[
        d := \frac{d_1 + \cdots + d_t}{t-1}.
    \]
    Assume the following conditions are satisfied:
    \begin{enumerate}
        \item The number $d$ is an integer.
        \item For all $i = 1,\ldots,t$, one has $d > d_i$.
        \item Each degree of a least common multiple of two distinct minimal generators of $I$ is at least $d$.
    \end{enumerate}
    Then the syzygy bundle of $I$ is semistable if  $R/I$ has the weak Lefschetz property.
\end{theorem}
\begin{proof}
    Consider the triangular region $T = T_d (I)$. By assumption (iii) and Lemma~\ref{lem:lcm-overlap}, we obtain that
    the punctures of $T$ do not overlap. Recall that the side length of the puncture to a minimal generator of degree
    $d_i$ is $d - d_i$. The definition of $d$ is equivalent to
    \[
        d = \sum_{i = 1}^t (d - d_i).
    \]
    We conclude that the region $T$ is balanced and perfectly-punctured. Combined with the weak Lefschetz property of
    $R/I$, the first property implies that $T$ is tileable by Proposition~\ref{pro:non-tileable-non-wlp}. Now
    Theorem~\ref{thm:tileable-semistable} gives the semistability of the syzygy bundle of $I$.
\end{proof}

The converse of the above result is not true, in general.

\begin{example} 
Consider the ideal 
\[
        J = (x^{7}, \ x^4 y^2 z^2, \ x y^3 z^3, \ y^{7},   \ z^{7}).
\]
Its syzygy bundle is semistable, and its triangular region $T_9 (J)$ is perfectly-puntured. However, $R/J$ does not have the weak Lefschetz property. Notice though that $R/I$ has the weak Lefschetz property, where $I$ is the very similar ideal 
\[
        I = (x^{7}, \ x^5 y z, \ x y^3 z^3, \ y^{7},   \ z^{8}). 
\]
Both regions, $T_9 (I)$ and $T_9 (J)$ are tileable and symmetric. In fact, they are examples of mirror symmetric regions that are studied  in~\cite{CN-mirror-symmetry}
\end{example}

Under stronger assumptions the converse to Theorem~\ref{thm:wlp-to-semistab} is indeed true.

\begin{theorem}\label{thm:wlp-iff-semistab}
    Let $I \subset R$ be an Artinian ideal with minimal monomial generators $m_1,\ldots,m_t$. Set
    \[
        d := \frac{d_1 + \cdots + d_t}{t-1},
    \]
    where $d_i = \deg m_i$. Assume the following conditions are satisfied:
    \begin{enumerate}
        \item The number $d$ is an integer.
        \item For all $i = 1,\ldots,t$, one has $d > d_i$.
        \item If $i \neq j$, then the degree of the least common multiple of $m_i$ and $m_j$ is at least $d+1$.
        \item If $m_i$ is not a power of $x, y$, or $z$, then $d - d_i$ is even.
    \end{enumerate}
    Then the syzygy bundle of $I$ is semistable if and only if  $R/I$ has the weak Lefschetz property.
\end{theorem}
\begin{proof}
    By Theorem~\ref{thm:wlp-to-semistab}, it is enough to show that $R/I$ has the weak Lefschetz property if the syzygy
    bundle of $I$ is semistable.

    Consider the region $T = T_d(I)$. In the proof of Theorem~\ref{thm:wlp-to-semistab} we showed that $T$ is balanced
    and perfectly-punctured. Hence $T$ is tileable by Theorem~\ref{thm:tileable-semistable}. Since all floating
    punctures of $T$ have an even side length by assumption (iv), Theorem~\ref{thm:pm-det} and
    Corollary~\ref{cor:same-sign} give that $Z(T)$ has maximal rank.

    Assumption (iii) means that the punctures of $T$ are not overlapping nor touching (see Lemma~\ref{lem:lcm-overlap}).
    Hence, Corollary~\ref{cor:socle-degree-bound} yields that the degrees of the socle generators of $R/I$ are at least
    $d-1$. Therefore, Theorem~\ref{thm:wlp-biadj} proves that $R/I$ has the weak Lefschetz property.
\end{proof}

\begin{remark}
   \label{rem:pictural description}
The assumptions of Theorem~\ref{thm:wlp-iff-semistab} have the following interpretation using the triangular region $T_d (I)$. Assumptions (ii) means that each minimal generator of $I$ gives a puncture of positive side length of $T_d (I)$. Condition (i) expresses the fact that the side lengths of all punctures add up to $d$, that is, $T_d (I)$ is perfectly punctured. Assumption (iv) says that all non-corner punctures have an even side length, and Condition (iii) is the requirement that no punctures of $T_d (I)$ touch or overlap.    
\end{remark}

We now show that, for all positive integers $d_1,\ldots,d_t$ with $t \geq 3$ that satisfy the numerical assumptions (i),
(ii), and (iv) of Theorem~\ref{thm:wlp-iff-semistab}, there is a monomial ideal $I$ whose minimal generators have
degrees $d_1,\ldots,d_t$ to which Theorem~\ref{thm:wlp-iff-semistab} applies and guarantees the weak Lefschetz property
of $R/I$.

\begin{example}\label{exa:ideals-with-wlp}
    Let $d_1,\ldots,d_t$ be $t \geq 3$ positive integers satisfying the following numerical conditions:
    \begin{enumerate}
        \item The number $d := \frac{d_1 + \cdots + d_t}{t-1}$ is an integer.
        \item For all $i = 1,\ldots,t$, one has $d > d_i$.
        %\item If $i \neq j$, then the degree of the least common multiple of $m_i$ and $m_j$ is at least $d+1$.
        \item At most three of the integers $d - d_i$ are not even.
    \end{enumerate}
    Re-indexing if needed, we may assume that $d_3 \leq \min \{d_1, d_2\}$ and that $d - d_i$ is even whenever
    $4 \leq i \leq t$. Consider the following ideal
    \begin{equation*}
        I = (x^{d_1}, y^{d_2}, z^{d_3}, m_4,\ldots,m_t),
    \end{equation*}
    where $m_4 = x^{d-d_3} y z^{-d-1 + d_3 +d_4}$ if $t \geq 4$, $m_5 = x^{2d - d_3 - d_4} y^2 z^{2d - 2 +d_3+d_4 + d_5}$
    if $t \geq 5$, and
    \begin{equation*}
        m_i =
        \begin{cases}
            x^{d-d_3} y^{1 + \sum_{k=4}^{i-1} (d- d_k)} z^{-d (i-3) - 1 + \sum_{k=3}^i d_k } & \text{if } 6 \leq i \leq t \text{ and $i$ is even}\\
            x^{-1 + \sum_{k=3}^{i-1} (d- d_k)} y^2 z^{-d (i-3) -1 + \sum_{k=3}^i d_k } & \text{if } 7 \leq i \leq t \text{ and $i$ is odd.}
        \end{cases}
    \end{equation*}
    Note that $\deg m_i = d_i$ for all $i$. One easily checks that the degree of the least common multiple of any two
    distinct minimal generators of $I$ is at least $d+1$, that is, the punctures of $T_d(I)$ do not overlap nor touch
    each other.
    \begin{figure}[!ht]
        \includegraphics[scale=1]{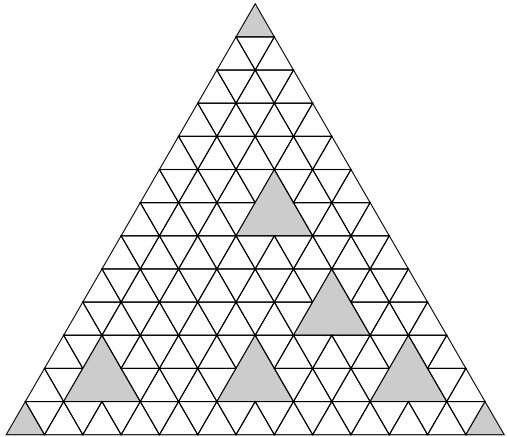}
        \caption{The region corresponding to $d_1 = d_2 = d_3 = 12$ and $d_4 = \cdots = d_8 = 11$ in Example~\ref{exa:ideals-with-wlp}.}
        \label{fig:example-d-13-t-8}
    \end{figure}
\end{example}

\begin{corollary}
    If $I$ is an ideal as defined in Example~\ref{exa:ideals-with-wlp}, then 
    $R/I$ has the weak Lefschetz property and the syzygy bundle of $I$ is semistable.
\end{corollary}
\begin{proof}
    By construction, the considered ideals satisfy assumptions (i)--(iv) of Theorem~\ref{thm:wlp-iff-semistab}.
    Furthermore, the region $T_d(I)$ has no over-punctured monomial subregions. Hence, it is tileable by
    Proposition~\ref{pro:pp-tileable}. (Alternatively, one can exhibit a family of non-intersecting lattice paths to check
    tileability.) By Theorem~\ref{thm:tileable-semistable}, it follows that the syzygy bundle of $I$ is semistable, and
    hence $R/I$ has the weak Lefschetz property by Theorem~\ref{thm:wlp-iff-semistab}.
\end{proof}

\begin{remark}
    Given an integer $t \geq 3$, there are many choices for the integers $d_1,\ldots,d_t$, and thus for the ideals
    exhibited in Example~\ref{exa:ideals-with-wlp}. A convenient choice, for which the description of the ideal becomes
    simpler, is $d_1 = 2t -4$, $d_2 = d_3 = d-1$, and $d_4 = \cdots = d_t = d-2$, where $d$ is any integer satisfying
    $d \geq 2t -3$. Then the corresponding ideal is
    \begin{equation*}
        I = (x^{2t - 4}, y^{d - 1}, z^{d-1}, xyz^{d-4}, x^3 y^2 z^{d-7}, m_6,\ldots,m_t ),
    \end{equation*}
    where
    \begin{equation*}
        m_i =
        \begin{cases}
            xy^{2i -7} z^{d+4 - 2i} & \text{if } 6 \leq i \leq t \text{ and $i$ is even}\\
            x^{2i - 8} y^2 z^{d+4 - 2i} & \text{if } 7 \leq i \leq t \text{ and $i$ is odd.}
        \end{cases}
    \end{equation*}
If $d = 13$ and $t = 8$, then this gives the ideal $I$ mentioned in the introduction. Its triangular region $T_{13} (I)$ is depicted in Figure~\ref{fig:example-d-13-t-8}.
\end{remark}

%--------------------
% -- The Bibliography

%---------------------------
% -- Happy, Happy, Joy, Joy!
\end{document}